\documentclass[dvipdfmx]{amsart}
\usepackage{amsmath, amsfonts, amssymb,amscd,graphics,graphicx,color,eucal,setspace} 
\usepackage{latexsym} 
\usepackage{color}

\usepackage{url}
\usepackage{amscd}
\usepackage{here}
\numberwithin{equation}{section}

\theoremstyle{plain}
\newtheorem{theo}{Theorem}[section]
\newtheorem{lemm}[theo]{Lemma} 
\newtheorem{coro}[theo]{Corollary}
\newtheorem{prop}[theo]{Proposition} 

\theoremstyle{definition}
\newtheorem*{rema}{Remark}

\def\C{\mathbb C}
\def\R{\mathbb R}

\def\Z{\mathbb Z}
\def\N{\mathbb P}
\def\T{T}
\def\Y{\mathfrak{Y}(h)}
\def\X{\mathfrak{X}(h)}

\def\v{\mathbf v}

\def\0{\mathbf 0}

\def\z{z}
\def\Xh{X(h)}
\def\Yh{Y(h)}
\def\Zh{Z(h)}

\def\fh{\mathcal{X}_h}
\def\gh{\mathcal{Y}_h}

\def\q{q}
\def\Sn{\mathfrak{S}_n}

\DeclareMathOperator{\Map}{Map}
\DeclareMathOperator{\Fl}{Fl}
\DeclareMathOperator{\U}{U}
\DeclareMathOperator{\ch}{ch}
\DeclareMathOperator{\sgn}{sgn}
\DeclareMathOperator{\asc}{asc}
\DeclareMathOperator{\LLT}{LLT}
\DeclareMathOperator{\Ind}{Ind}

\def\Fln{\Fl(n)}

\begin{document}
\title[LLT polynomials and twin]{Unicellular LLT polynomials and \\
twin of regular semisimple Hessenberg varieties}

\author[M. Masuda]{Mikiya Masuda}
\address{Osaka Central Advanced Mathematical Institute, Sumiyoshi-ku, Osaka 558-8585, Japan.}
\email{masuda@sci.osaka-cu.ac.jp}

\author[T. Sato]{Takashi Sato}
\address{Osaka Central Advanced Mathematical Institute, Sumiyoshi-ku, Osaka 558-8585, Japan.}
\email{00tkshst00@gmail.com}

\date{\today}


\date{\today}

\begin{abstract}
The solution of Shareshian-Wachs conjecture by Brosnan-Chow linked together the cohomology of regular semisimple Hessenberg varieties and graded chromatic symmetric functions on unit interval graphs. On the other hand, it is known that unicellular LLT polynomials have similar properties to graded chromatic symmetric functions. In this paper, we link together the unicellular LLT polynomials and twin of regular semisimple Hessenberg varieties introduced by Ayzenberg-Buchstaber. We prove their palindromicity from topological viewpoint. We also show that modules of a symmetric group generated by faces of a permutohedron are related to a shifted unicellular LLT polynomial and observe the $e$-positivity of shifted unicellular LLT polynomials, which is established by Alexandersson-Sulzgruber in general, for path graphs and complete graphs through the cohomology of the twins. 
\end{abstract}

\maketitle

\setcounter{tocdepth}{1}

\section{Introduction}

A regular semisimple Hessenberg variety is a subvariety of the flag variety $\Fln$ defined as
\[
\Xh:=\{(V_1\subset \cdots\subset V_n=\C^n)\in \Fln\mid SV_j\subset V_{h(j)} \quad(\forall j\in [n])\}
\]
where $S$ is a linear endomorphism of $\C^n$ with distinct eigenvalues and $h$ is a function (called a Hessenberg function) from $[n]=\{1,\dots,n\}$ to itself satisfying 
\[
h(1)\le h(2)\le \cdots \le h(n)\quad \text{and}\quad h(j)\ge j\quad (\forall j\in [n]).
\]
It is smooth and its topology is independent of the choice of $S$. The space $\Xh$ may not have a symmetry of the symmetric group $\mathfrak{S}_n$ of order $n$ but its cohomology $H^*(\Xh)$ becomes a graded $\mathfrak{S}_n$-module under the dot action introduced by Tymoczko \cite{tymo08}. Throughout this paper, the coefficient of cohomology will be the complex numbers $\C$ unless otherwise stated. 

The Hessenberg function $h$ defines a unit interval graph $G_h$. For $G_h$, Shareshian-Wachs \cite{sh-wa16} introduced a polynomial $X_h(\z;q)$ in $q$ with symmetric functions in $z_1,z_2,\dots$ as coefficients, which refines Stanley's chromatic symmetric function by introducing grading. Remarkably, the solution of Shareshian-Wachs conjecture by Brosnan-Chow \cite{br-ch18} (see also \cite{guay16}) linked these two objects in geometry and combinatorics together by showing the identity 
\begin{equation} \label{eq:intro_Xh}
X_{h}(\z;q)=\sum_{i=0}^{|h|}{\ch}(H^{2i}(\Xh)\otimes \mathbf{U})q^i
\end{equation}
where $|h|=\sum_{j=1}^n(h(j)-j)=\dim_\C\Xh$, $\ch$ denotes the Frobenius characteristic sending $\mathfrak{S}_n$-modules to symmetric functions of degree $n$, and $\mathbf{U}$ is the alternating complex $1$-dimensional $\mathfrak{S}_n$-module. 

On the other hand, an LLT polynomial is defined for an $n$-tuple of skew Young diagrams, which can be seen as a $q$-deformation of products of skew Schur functions. An LLT polynomial is called \emph{unicellular} if each skew Young diagram indexing the LLT polynomial is a single box. In fact, we can think of a unicellular LLT polynomial as being associated to a Hessenberg function $h$, so we denote it by $\LLT_h(\z;q)$. It is a polynomial in $q$ with symmetric functions in $z_1,z_2,\dots$ as coefficients similarly to $X_h(\z;q)$. 

The polynomials $X_h(\z;q)$ and $\LLT_h(\z;q)$ have different origins but they have many similar properties. One reason is that they are related through the equivariant cohomology $H_T^*(\Xh)$ of $\Xh$ where $T$ is a compact torus acting naturally on $\Xh$. Indeed, $H^*_T(\Xh)$ is a free module over a polynomial ring $H^*(BT)=\C[t_1,\dots,t_n]$, where $BT$ is the classifying space $BT$ of $T$, and the quotient ring 
\[
H^*(\Xh)=H^*_T(\Xh)/(t_1,\dots,t_n)
\]
corresponds to $X_h(\z;q)$ by \eqref{eq:intro_Xh}. On the other hand, $H^*_T(\Xh)$ is a free module over another polynomial ring $\C[x_1,\dots,x_n]$, where $x_1,\dots,x_n$ are a linear system of parameters for $H^*_T(\Xh)$ (see Remark~\ref{rema:meaning_of_xi} for the geometrical meaning of $x_i$'s), and the quotient ring 
\begin{equation} \label{eq:intro_quotient_Yh}
H^*_T(\Xh)/(x_1,\dots,x_n)
\end{equation}
corresponds to $\LLT_h(\z;q)$.

The quotient ring \eqref{eq:intro_quotient_Yh}, which is also a graded $\mathfrak{S}_n$-module, is a Poincar\'e duality algebra. We observe that it agrees with the cohomology ring of the \emph{twin} $\Yh$ of $\Xh$ introduced by Ayzenberg-Buchstaber \cite{ay-bu21}\footnote{$\Xh$ is denoted by $Y_h$ while $\Yh$ is denoted by $X_h$ in \cite{ay-bu21}.}, even as graded $\mathfrak{S}_n$-modules with a naturally defined $\mathfrak{S}_n$-action on $H^*(\Yh)$. Consequently, we have 
\begin{equation} \label{eq:intro_Yh}
\LLT_{h}(\z;q)=\sum_{i=0}^{|h|}{\ch}(H^{2i}(\Yh))q^i, 
\end{equation}
which is a counterpart of \eqref{eq:intro_Xh}. 
The twin $\Yh$ resembles $\Xh$, e.g.\ they have the same Poincar\'e polynomial but their cohomology rings are not isomorphic in general and $\Yh$ is not an algebraic variety in general although it is a compact smooth $T$-manifold.

When $h=(2,3,\dots,n,n)$, i.e.\ $h(j)=j+1$ for $j\in [n-1]$ and $h(n)=n$, the orbit space $\Xh/T=\Yh/T$ is a permutohedron $\Pi_n$ of dimension $n-1$. The $f$-polynomial of $\Pi_n$ is defined as 
\[
f_{\Pi_n}(q):=\sum_{i=0}^{n-1}f_i(\Pi_n)q^i
\]
where $f_i(\Pi_n)$ denotes the number of $i$-dimensional faces of $\Pi_n$, and it is well-known that the Poincar\'e polynomial of $\Xh$ and $\Yh$ agrees with the $h$-polynomial $h_{\Pi_n}(q):=f_{\Pi_n}(q-1)$ of $\Pi_n$. 
The permutohedron $\Pi_n$ sits in $\R^n$ and has symmetry of $\mathfrak{S}_n$ under permuting the coordinates of $\R^n$. So, the $i$-dimensional faces of $\Pi_n$ generate an $\mathfrak{S}_n$-module, denoted $F_i(\Pi_n)$. 
Therefore, it is natural to ask whether the $\mathfrak{S}_n$-module version of the $f$-polynomial, that is $\sum_{i=0}^{n-1}F_i(\Pi_n)q^i$, is related to $X_h(\z;q+1)$ or $\LLT_h(\z;q+1)$. It turns out that 
\begin{equation} \label{eq:intro_face_module}
\sum_{i=0}^{n-1}\ch(F_i(\Pi_n)\otimes\mathbf{U})q^i=\LLT_h(\z;q+1). 
\end{equation} 
This formula provides a geometric meaning of $\LLT_h(\z;q+1)$ for $h=(2,3,\dots,n,n)$. Related to this, the $e$-positivity of $\LLT_h(\z;q+1)$ is established in \cite{al-su22, dadd20} for any Hessenberg function $h$ (see also \cite{alex21}, \cite{al-pa18}). We compute $\LLT_h(\z;q)$ for $h=(2,3,\dots,n,n)$ and $h=(n,\dots,n)$ through the cohomology of the twin $\Yh$ and observe the $e$-positivity of $\LLT(\z;q+1)$ for such $h$. 

The paper is organized as follows. In Section~\ref{sect:2} we review the GKM description of the equivariant cohomology $H^*_T(\Xh)$ and the dot action on $H^*_T(\Xh)$. In Section~\ref{sect:3} we discuss the relation between chromatic symmetric functions and unicellular LLT polynomials through the equivariant cohomology. In Section~\ref{sect:4} we recall the twin $\Yh$ and its equivariant cohomology. We also define an action of $\mathfrak{S}_n$ on $H^*(\Yh)$ which is a modification of the dot action. The polynomials $X_h(\z;q)$ and $\LLT_h(\z;q)$ are known to be palindromic. In section~\ref{sect:5} we prove the palindromicity using the Atiyah-Bott localization formula. We prove \eqref{eq:intro_face_module} and observe the $e$-positivity of $\LLT_h(\z;q+1)$ for $h=(2,3,\dots,n,n)$ in Section~\ref{sect:6} and for $h=(n,\dots,n)$ in Section~\ref{sect:7} through the cohomology of the twin $\Yh$.

\section{Regular semisimple Hessenberg variety} \label{sect:2}

A flag variety $\Fln$ is the variety consisting of complete flags in $\C^n$:
\[
\Fln:=\{(V_1\subset \cdots\subset V_n=\C^n)\mid \dim_\C V_i=i\quad (\forall i\in [n])\}. 
\]
Let $h\colon [n]\to [n]$ be a function (called a Hessenberg function) satisfying 
\[
h(1)\le h(2)\le \cdots \le h(n)\quad \text{and}\quad h(j)\ge j\quad (\forall j\in [n]).
\]
The regular semisimple Hessenberg variety $\Xh$ is a subvariety of $\Fln$ defined by 
\begin{equation} \label{eq:definition_Xh}
\Xh:=\{(V_1\subset \cdots\subset V_n=\C^n)\in \Fln\mid SV_j\subset V_{h(j)} \quad(\forall j\in [n])\}
\end{equation}
where $S$ is a linear endomorphism of $\C^n$ with distinct eigenvalues. We think of $S$ as a square matrix of size $n\times n$ in the following. 

\begin{prop}[\cite{ma-pr-sh92}] \label{prop:Xh_property}
The following holds. 
\begin{enumerate}
\item $\Xh$ is smooth. 
\item $\dim_\C \Xh=\sum_{j=1}^n(h(j)-j)$.
\item $H^{odd}(\Xh)=0$. 
\end{enumerate}
\end{prop}

The definition of $\Xh$ depends on the choice of $S$ but the diffeomorphism type of $\Xh$ does not, so $S$ is suppressed from the notation $\Xh$. We take $S$ to be a diagonal matrix with distinct diagonal entries in the following and further assume $S$ to be a real matrix in Section~\ref{sect:4}. 

Let $\T$ denote the $S^1$-torus consisting of diagonal matrices in the unitary group $\U(n)$.  Since $T$ commutes with the diagonal matrix $S$, the natural action of $\T$ on $\Fln$ preserves $X(h)$ and one can see that 
\begin{equation} \label{eq:fixed_set}
\Xh^{\T}=\Fln^{\T}=\mathfrak{S}_n
\end{equation}
where the symmetric group $\mathfrak{S}_n$ of order $n$ is identified with the permutation flags in $\Fln$. 

\subsection{Equivariant cohomology} 
The equivariant cohomology of $\Xh$ is defined as
\[
H^*_{\T}(\Xh):=H^*(E\T\times_\T\Xh)
\]
where $E\T$ is the total space of the universal principal $\T$-bundle $E\T\to B\T=E\T/\T$ and $E\T\times_\T X$ is the orbit space of $E\T\times X$ by the $\T$-action defined by $(u,x)\to (ug^{-1},gx)$. 
For a point $pt$, we have $pt\times_\T E\T=B\T$; so 
\[
H^*_\T(pt)=H^*(B\T)
\]
which is a polynomial ring in $n$ elements in $H^2(B\T)$ explained below. 

Let $\pi_i\colon T\to S^1$ denote the projection to the $(i,i)$-entry of the diagonal torus $T$ and $\C(\pi_i)$ the complex 1-dimensional $\T$-module defined by $\pi_i$. The projection $E\T\times \C(\pi_i)\to E\T$ on the first factor induces a complex line bundle 
\begin{equation} \label{eq:ti}
E\T\times_\T \C(\pi_i)\to E\T/\T=B\T. 
\end{equation}
Although the complex line bundle $E\T\times \C(\pi_i)\to E\T$ is trivial, the line bundle \eqref{eq:ti} is nontrivial. We take the first Chern class of \eqref{eq:ti} and set 
\[
t_i:=c_1(E\T\times_\T \C(\pi_i)\to B\T)\in H^2(B\T)\quad\text{for $i=1,\dots,n$}. 
\]
As is well-known, $H^*(B\T)$ is a polynomial ring in $t_1,\dots,t_n$, i.e. 
\[
H^*(B\T)=\C[t_1,\dots,t_n].
\]

The projection $ET\times \Xh\to ET$ on the first factor induces a fiber bundle 
\[
ET\times_T \Xh\xrightarrow{\pi} ET/T=BT
\]
with $\Xh$ as a fiber. Through $\pi^*\colon H^*(BT)\to H^*_T(\Xh)$, the equivariant cohomology $H^*_T(\Xh)$ becomes a module over $H^*(BT)$. Since $H^{odd}(\Xh)=0$, this module is free and 
\begin{equation} \label{eq:HXh}
H^*_T(\Xh)=H^*(BT)\otimes H^*(\Xh) \quad\text{as graded $H^*(BT)$-modules},
\end{equation}
and 
\begin{equation} \label{eq:HXh2}
H^*(\Xh)= H^*_T(\Xh)/(t_1,\dots,t_n)\quad\text{as graded rings}
\end{equation}
where $(t_1,\dots,t_n)$ denotes the ideal in $H^*_T(\Xh)$ generated by $\pi^*(t_1),\dots,\pi^*(t_n)$. 

Remember that $\Xh^\T=\mathfrak{S}_n$. Since $H^{odd}(\Xh)=0$, the restriction map 
\[
\begin{split}
H^*_T(\Xh)\hookrightarrow H^*_T(\Xh^T)=\bigoplus_{w\in\mathfrak{S}_n}H^*(BT)&=
\bigoplus_{w\in\mathfrak{S}_n}\C[t_1,\dots,t_n]\\
&=\Map(\mathfrak{S}_n,\C[t_1,\dots,t_n]) 
\end{split}
\]
is injective. The image is described as follows. 

\begin{prop}[\cite{tymo08}] \label{prop:equivariant_cohomology_Xh}
Let $f\in \Map(\mathfrak{S}_n,\C[t_1,\dots,t_n])$. 
Then $f\in H^*_T(\Xh)$ if and only if 
\[
f(v)\equiv f(w)\pmod{t_{w(i)}-t_{w(j)}} \quad \text{whenever $v=w(i,j)$ for $j<i\le h(j)$}. 
\]
\end{prop} 

For each $i\in [n]$, $x_i\in\Map(\mathfrak{S}_n,\C[t_1,\dots,t_n])$ defined by
\begin{equation} \label{eq:x_i}
x_i(w):=t_{w(i)}
\end{equation}
obviously satisfies the congruence relation in Proposition~\ref{prop:equivariant_cohomology_Xh}, so it lies in $H^2_T(\Xh)$. 

\begin{rema} \label{rema:meaning_of_xi}
The geometrical meaning of $x_i$ is the following. There is a nested sequence of tautological vector bundles over $\Fl(n)$:
\[
E_1\subset \cdots \subset E_n=\Fln\times \C^n
\]
where 
\[
E_i:=\{((V_1\subset\cdots\subset V_n),v)\in \Fln\times \C^n\mid v\in V_i\}\quad (i=1,\dots,n).
\]
They are all complex $T$-vector bundles. Then, $x_i$ is the restriction of the equivariant first Chern class of the quotient $T$-line bundle $E_i/E_{i-1}$ to $\Xh$, where $E_0=\Fl(n)$. 
\end{rema}

\subsection{Dot action}
We review Tymoczko's dot action. For each polynomial $p\in \C[t_1,\dots,t_n]$ and $\sigma\in\mathfrak{S}_n$, define 
\[
(\sigma p)(t_1,\dots,t_n):=p(t_{\sigma(1)},\dots,t_{\sigma(n)}).
\]
Then, the dot action of $\sigma\in \mathfrak{S}_n$ on $f\in \Map(\mathfrak{S}_n,\C[t_1,\dots,t_n])$ is defined as 
\begin{equation} \label{eq:dot_action}
(\sigma\cdot f)(w):=\sigma(f(\sigma^{-1}w)).
\end{equation}

It follows from Proposition~\ref{prop:equivariant_cohomology_Xh} that the dot action of $\mathfrak{S}_n$ on $\Map(\mathfrak{S}_n,\C[t_1,\dots,t_n])$ preserves the subring $H^*_T(\Xh)$. Moreover, one can see from \eqref{eq:x_i} and \eqref{eq:dot_action} that 
\begin{equation} \label{eq:dot_action_on_t_and_x}
\sigma\cdot t_i=t_{\sigma(i)},\quad \sigma\cdot x_i=x_i\qquad \text{for $\sigma\in\mathfrak{S}_n$},
\end{equation}
where $t_i$ is regarded as a constant map. Therefore, the dot action on $H^*_T(\Xh)$ preserves the ideals $(t_1,\dots,t_n)$ and $(x_1,\dots,x_n)$ in $H^*_T(\Xh)$ and induces actions of $\mathfrak{S}_n$ on quotient rings
\begin{equation} \label{eq:definition_L}
\X:=H^*_T(\Xh)/(t_1,\dots,t_n)\quad\text{and}\quad \Y:=H^*_T(\Xh)/(x_1,\dots,x_n).
\end{equation}
In fact, $\X$ is $H^*(\Xh)$ by \eqref{eq:HXh2}. In Section~\ref{sect:4} we will observe that $\Y$ can be understood as the cohomology of the twin of $\Xh$. 

Since $H^*_T(\Xh)$ contains the polynomial ring $\C[t_1,\dots,t_n]$ (through $\pi^*$) and is a finite dimensional free module over it by \eqref{eq:HXh}, $H^*_T(\Xh)$ is a Cohen-Macaulay ring of Krull dimension $n$. 
On the other hand, one can easily see that there is no algebraic relations among $x_1,\dots,x_n$, so that $H^*_T(\Xh)$ contains a polynomial ring $\C[x_1,\dots,x_n]$. Since $x_1,\dots,x_n$ is a linear system of parameter of the Cohen-Macaulay ring $H^*_T(\Xh)$ of the Krull-dimension $n$, $H^*_T(\Xh)$ is a free module over $\C[x_1,\dots,x_n]$ and the quotient ring $\Y$ in \eqref{eq:definition_L} is of finite dimension.

\begin{lemm} \label{lemm:Qh}
The equivariant cohomology $H^*_T(\Xh)$ has the following two presentations as graded $\mathfrak{S}_n$-modules: 
\[
\C[t_1,\dots,t_n]\otimes \X=H^*_T(\Xh)=\C[x_1,\dots,x_n]\otimes \Y 
\]
where $\mathfrak{S}_n$ permutes $t_i$'s but fixes $x_i$'s. 
\end{lemm}

\begin{proof}
Since the quotient map $H^*_T(\Xh)\to \X$ is $\mathfrak{S}_n$-equivariant, there is an $\mathfrak{S}_n$-submodule $U$ of $H^*_T(\Xh)$ which maps to $\X$ isomorphically. Then the multiplication $H^*(BT)\otimes U\to H^*_T(\Xh)$ gives an isomorphism as graded $\mathfrak{S}_n$-modules. This proves the former identity. The latter identity can be proved by the same argument. 
\end{proof}

\section{Chromatic symmetric functions and unicellular LLT polynomials} \label{sect:3}

In this section, we recall some known results on the relation between chromatic symmetric functions and unicellular LLT polynomials for the reader's convenience, see \cite[Sections 6.6 and 6.7]{al-su22} for more information. 

Let $R(\mathfrak{S}_n)$ be the complex representation ring of $\mathfrak{S}_n$, $\Lambda_n$ the ring of symmetric functions of degree $n$ in infinitely many variables $z_1,z_2, \dots$, and 
\[
\ch\colon R(\mathfrak{S}_n)\to \Lambda_n
\]
the Frobenius characteristic, which is a group isomorphism (see \cite{fult97} for details). 

For a graded $\mathfrak{S}_n$-module $A=\bigoplus_{i=0}^\infty A_{2i}$, we define 
\[
R(A;q):=\sum_{i=0}^\infty A_{2i}q^i \in R(\mathfrak{S}_n)[[q]]
\]
and the Frobenius series of $A$ is the formal power series 
\[
F_A(\z;q):=\sum_{i=0}^\infty {\rm ch}(A_{2i})q^i\in \Lambda_n[[q]]. 
\]

We set 
\[
\fh(\z;q):=F_{\X}(\z;q)\quad\text{and}\quad \gh(\z;q):=F_{\Y}(\z;q).
\]
\begin{prop} \label{prop:1}
We have
\begin{enumerate}
\item $\displaystyle{\gh(\z;q)=(1-q)^n \fh\left[\frac{Z}{1-q};q\right]}$
\item $\displaystyle{\fh(\z;q)=(1-q)^{-n}\gh\left[(1-q)Z;q\right]}$
\end{enumerate}
where $Z=z_1+z_2+\cdots$ and $[\ \ ]$ denotes the plethystic substitution. 
\end{prop}

\begin{proof}
By Lemma~\ref{lemm:Qh}, we have
\begin{equation} \label{eq:XhYh}
\C[t_1,\dots,t_n]\otimes \X= \C[x_1,\dots,x_n]\otimes \Y
\end{equation}
as graded $\mathfrak{S}_n$-modules. Let $V$ be the defining representation $\C^n$ of $\mathfrak{S}_n$. Then 
\[
\C[t_1,\dots,t_n]= S^*V=\bigoplus_{k=0}^\infty S^kV\quad\text{as graded $\mathfrak{S}_n$-modules},
\] 
where $S^kV$ denotes the $k$-th symmetric product of $V$.  Therefore, taking $R(\ ;q)$ on the both sides of \eqref{eq:XhYh}, we obtain
\begin{equation} \label{eq:XhYh_R}
R(S^*V;q)R(\X;q)=(1-q)^{-n}R(\Y;q).
\end{equation}

(1) It follows from \eqref{eq:XhYh_R} that 
\begin{equation*} \label{eq:leftF}
\sum_{k=0}^\infty q^k F_{S^kV\otimes\X}(\z;q)=(1-q)^{-n}F_{\Y}(\z;q).
\end{equation*}
Here the left side above is $\fh[\frac{Z}{1-q};q]$ by \cite[Proposition 3.3.1]{haim02}, so (1) follows. 

(2) Since 
\[
R(S^*V;q)R(\Lambda^*V;-q)=1,
\]
where $\Lambda^*V=\bigoplus_{k=0}^n\Lambda^k V$ and $\Lambda^kV$ is the $k$-th exterior product of $V$, it follows from \eqref{eq:XhYh_R} that 
\[
R(\X;q)=(1-q)^{-n}R(\Lambda^*V;-q)R(\Y;q)
\]
and this implies 
\[
F_{\X}(\z;q)=(1-q)^{-n}\sum_{k=0}^\infty (-q)^kF_{\Lambda^kV\otimes \Y}(\z;q).
\]
Here the sum at the right hand side above is $\gh[(1-q)Z;q]$ by \cite[Proposition 3.3]{haim02}, so (2) follows. 
\end{proof}

\subsection{Chromatic symmetric functions} 
Let $G$ be a graph with vertex set $[n]$. Given a vertex coloring $\kappa\colon G\to \N$, where $\N$ denotes the set of positive integers, let $\asc{(\kappa)}$ denotes the number of edges $\{i,j\}$ of $G$ such that $\kappa(i)<\kappa(j)$ when $i<j$. Then the chromatic quasisymmetric function of $G$ is defined as 
\[
X_G(\z;q):=\sum_{\kappa\colon G\to \N,\, \text{proper}}z_{\kappa(1)}\cdots z_{\kappa(n)}q^{\asc(\kappa)}
\]
where the sum runs over all proper colorings of the graph $G$. Here, a coloring is called proper if there are no edges whose endpoints receive the same color. The function $X_G(\z;q)$ is quasisymmetric in $\z$ but if $G$ is a unit interval graph $G_h$ associated to $h$, then it is symmetric in $\z$, and we set 
\[
X_h(\z;q):=X_{G_h}(\z;q).
\]
\begin{coro} \label{coro:XGh}
$X_{h}(\z;q)=(q-1)^{-n}\gh[(q-1)Z;q].$
\end{coro} 

\begin{proof}
Let $\omega$ be an involution on $\Lambda_n$ switching elementary symmetric functions $e_\lambda$ and complete symmetric functions $h_\lambda$ for any partition $\lambda$ of $n$ (see \cite{fult97} for details). 
Applying the involution $\omega$ to the both sides of (2) in Proposition~\ref{prop:1}, we obtain 
\[
\omega \fh(\z;q)=(1-q)^{-n}\omega \gh\left[(1-q)Z;q\right]. 
\]
Here the left hand side above is $X_{G_h}(\z;q)$ by the theorem of Brosnan-Chow \cite{br-ch18} while the right hand side above agrees with the right hand side in the corollary because $\varphi[-Z]=(-1)^n\omega \varphi[Z]$ for a homogeneous symmetric function $\varphi$ of degree $n$. 
\end{proof}

\subsection{Unicellular LLT polynomials}
The unicellular LLT polynomials is a subset of LLT polynomials indexed by tuples of skew shapes, such that each shape is a single box. But it can be defined as
\[
\LLT_h(\z;q)=\sum_{\kappa\colon G_h\to \N}z_{\kappa(1)}\cdots z_{\kappa(n)}q^{\asc(\kappa)}
\]
similarly to the chromatic symmetric functions $X_{h}(\z;q)$ (see \cite{ca-me17}, \cite{al-pa18}). The only difference is that we do not require that colorings $\kappa$ be proper for $\LLT_h(\z;q)$. 
The chromatic symmetric function $X_{h}(\z;q)$ has the following interpretation
\[
X_{h}(\z;q)=\omega\fh(\z;q)
\]
as mentioned in the proof of Corollary~\ref{coro:XGh}. Similarly, the unicellular LLT polynomial $\LLT_h(\z;q)$ has the following interpretation. 

\begin{prop} \label{prop:LLT}
$\LLT_h(\z;q)=\gh(\z;q)$.
\end{prop}

\begin{proof}
Carlson-Mellit \cite[Prop.3.4]{ca-me17} proves that 
\[
X_{h}(\z;q)=(q-1)^{-n}\LLT_h[(q-1)Z;q].
\]
This together with Corollary~\ref{coro:XGh} implies the proposition. 
\end{proof}

\section{Twin of regular semisimple Hessenberg variety} \label{sect:4}

The quotient ring $\X=H^*_T(\Xh)/(t_1,\dots,t_n)$ agrees with $H^*(\Xh)$, so it is natural to expect that $\Y=H^*_T(\Xh)/(x_1,\dots,x_n)$ is the cohomology ring of some $T$-manifold. In this section we observe that $\Y$ is indeed the cohomology ring of the twin of $\Xh$ introduced by Ayzenberg-Buchstaber \cite{ay-bu21}. 

As is well-known, the flag variety $\Fln$ can be identified with the homogeneous space $\U(n)/T$. Indeed, to an element 
$g=[\v_1,\dots,\v_n]\in \U(n)$, we associate a complete flag 
\[
\left(\langle \v_1\rangle \subset \langle \v_1,\v_2\rangle \subset \cdots \subset \langle \v_1,\dots,\v_n\rangle\right)\in \Fln 
\]
where $\langle \ \rangle$ denotes the complex vector space generated by elements in $\langle\ \rangle$, and this correspondence induces the identification
\[
\Fln= \U(n)/T.
\] 
Through this identification, \eqref{eq:definition_Xh} can be written as
\begin{equation} \label{eq:Xh}
\Xh=\{ gT\in \U(n)/T\mid g^{-1}Sg\in M_h\},
\end{equation}
where the diagonal matrix $S$ with distinct diagonal entries is assumed to be real and $M_h$ denotes the vector space of all Hermitian matrices with vanishing $(i,j)$ entries for $i>h(j)$ or $j>h(i)$ (note that since $S$ is assumed to be a real matrix, $g^{-1}Sg$ is a Hermitian matrix).

\subsection{Twin of $\Xh$} 
The twin of $\Xh$, denoted by $\Yh$, is defined as
\begin{equation} \label{eq:Yh}
\Yh:=\{gT\in \U(n)/T\mid gSg^{-1}\in M_h\}. 
\end{equation}

\begin{lemm}[\cite{ay-bu21}]
The twin $\Yh$ is a smooth manifold. 
\end{lemm}

\begin{proof} 
This is proved in \cite{ay-bu21} in two ways. One uses Toda flow. We shall give the other proof for the reader's convenience. Let $p\colon \U(n)\to \U(n)/T$ be the quotient map and consider
\[
\Zh:=p^{-1}(\Xh).
\]
We see from \eqref{eq:Xh} and \eqref{eq:Yh} that if $\iota\colon \U(n)\to \U(n)$ denotes the diffeomorphism sending $g$ to $g^{-1}$, then $\Yh=\iota(\Zh)/T$. Therefore, it suffices to see that $\Zh$ is a smooth manifold but it is so because $\Xh$ is a smooth submanifold of $\U(n)/T$ and $p\colon \U(n)\to \U(n)/T$ is a principal $T$-bundle. 
\end{proof}

\begin{rema}
As pointed out in \cite{ay-bu21}, the twin $\Yh$ may be regarded as the quotient $T\backslash \Zh$ of $\Zh$ by the \emph{left} $T$-action on $\U(n)$. Indeed, the diffeomorphism $\iota$ above induces a diffeomorphism from $T\backslash \U(n)$ to $\U(n)/T$ and it maps $T\backslash \Zh$ onto $\iota(\Zh)/T=\Yh$. 
\end{rema}

Apparently, the twins $\Xh$ and $\Yh$ look similar. Indeed, both have vanishing odd degree cohomology and have the same Betti numbers. However, their cohomology rings are not isomorphic in general, see \cite[Example 3.12]{ay-bu21}. 

It follows from the definition \eqref{eq:Yh} that $\Yh$ is invariant under the left $T$-action on $\U(n)/T$. Similarly to $\Xh$, we have 
\[
\Yh^T=\mathfrak{S}_n
\]
and since $H^{odd}(\Yh)=0$ (\cite{ay-bu21}), the restriction map 
\[
H^*_T(\Yh)\hookrightarrow H^*_T(\Yh^T)=\bigoplus_{w\in\mathfrak{S}_n}H^*(BT)=\Map(\mathfrak{S}_n,\C[t_1,\dots,t_n])
\]
is injective. The following description of $H^*_T(\Yh)$ is slightly different from Proposition~\ref{prop:equivariant_cohomology_Xh} for $H^*_T(\Xh)$. 

\begin{prop}[\cite{ay-bu21}] \label{prop:equivariant_cohomology_Yh}
Let $f\in \Map(\mathfrak{S}_n,\C[t_1,\dots,t_n])$. Then $f\in H^*_T(\Yh)$ if and only if 
\[
f(v)\equiv f(w)\pmod{t_{i}-t_{j}} \quad \text{whenever $v=w(i,j)$ for $j<i\le h(j)$}. 
\]
\end{prop} 

As noticed in \cite{ay-bu21}, $H^*_T(\Xh)$ and $H^*_T(\Yh)$ are isomorphic as \emph{rings}. Indeed, the automorphism $\xi$ of $\Map(\mathfrak{S}_n,\C[t_1,\dots,t_n])$ defined by
\begin{equation} \label{eq:xi}
(\xi(f))(w):=w(f(w))
\end{equation}
induces a ring isomorphism 
\begin{equation} \label{eq:xi_iso}
\xi\colon H^*_T(\Yh)\xrightarrow{\cong} H^*_T(\Xh)
\end{equation}
by Propositions~\ref{prop:equivariant_cohomology_Xh} and~\ref{prop:equivariant_cohomology_Yh}. 
Since $(\xi(t_i))(w)=w(t_i(w))=w(t_i)=t_{w(i)}$, we have 
\begin{equation} \label{eq:x_i2}
\xi(t_i)=x_i
\end{equation}
by \eqref{eq:x_i}. 

We have the dot action on $\Map(\mathfrak{S}_n,\C[t_1,\dots,t_n])$. 
We consider another action of $\mathfrak{S}_n$ on $\Map(\mathfrak{S}_n,\C[t_1,\dots,t_n])$ defined as
\begin{equation} \label{eq:dagger}
(\sigma\dagger f)(w):=f(\sigma^{-1}w),
\end{equation}
which we call dagger action. 
Then we have 
\begin{equation} \label{eq:dagger_dot}
\xi(\sigma\dagger f)=\sigma\cdot \xi(f)
\end{equation}
because
\[
\begin{split}
(\xi(\sigma\dagger f))(w)&=w((\sigma\dagger f)(w))=w(f(\sigma^{-1}w))=\sigma\sigma^{-1}w(f(\sigma^{-1}w))\\
&=\sigma(\xi(f)(\sigma^{-1}w))=(\sigma\cdot\xi(f))(w).
\end{split}
\]
Moreover, the congruence relation in Proposition~\ref{prop:equivariant_cohomology_Yh} is preserved under the dagger action. Therefore, $H^*_T(\Yh)$ has the dagger action. Since each $t_i$ is fixed under the dagger action, the dagger action makes 
\[
H^*(\Yh)=H^*_T(\Yh)/(t_1,\dots,t_n)
\]
a graded $\mathfrak{S}_n$-module. With this understanding, we have the following. 

\begin{prop} \label{prop:L_Yh}
$H^*(\Yh)$ (with the dagger action of $\mathfrak{S}_n$) is isomorphic to $\Y$ defined in \eqref{eq:definition_L} as graded rings and also as graded $\mathfrak{S}_n$-modules. Therefore, 
\[
R(H^*_T(\Xh);q)=\frac{1}{(1-q)^n}R(H^*(\Yh);q)
\]
by \eqref{eq:XhYh} and 
\[
\sum_{i=0}^\infty {\rm ch}(H^{2i}(\Yh))q^i=\LLT_h(\z;q)
\]
by Proposition~\ref{prop:LLT}. 
\end{prop}

\begin{proof}
The automorphism $\xi$ induces the desired isomorphism by \eqref{eq:xi_iso}, \eqref{eq:x_i2}, and \eqref{eq:dagger_dot}. 
\end{proof}

\section{Palindromicity and Poincar\'e duality} \label{sect:5}

Let $|h|=\sum_{j=1}^{n}(h(j)-j)$. Then $X_{h}(\z;q)$ and $\LLT_{h}(\z;q)$ are polynomials in $q$ of degree $|h|$ and it is known that they satisfy the following palindromicity:
\begin{equation} \label{eq:palindromicity}
q^{|h|}X_{h}(\z;q^{-1})=X_{h}(\z;q)\quad\text{and}\quad q^{|h|}\LLT_h(\z;q^{-1})=\omega \LLT_h(\z;q)
\end{equation}
(see \cite[Proposition 2.6]{sh-wa16} for the former and \cite[Lemma 4.1]{al-pa18} for the latter). 
In this section we observe that the palindromicity \eqref{eq:palindromicity} is a consequence of the Poincar\'e duality for $\Xh$. 

We treat the former identity in \eqref{eq:palindromicity} first. Remember that 
\[
R(H^*(\Xh);q)=\sum_{i=0}^{|h|}H^{2i}(\Xh)q^i
\]
and 
\[
\fh(\z;q)=\sum_{i=0}^{|h|}\ch(H^{2i}(\Xh))q^i,\quad \quad X_h(\z;q)=\omega \fh(\z;q). 
\]
Therefore, the following proposition is equivalent to the former identity in \eqref{eq:palindromicity}. 

\begin{prop} \label{prop:palindromicity_Xh}
$q^{|h|}R(H^*(\Xh);q^{-1})=R(H^*(\Xh);q)$, 
namely $H^{2k}(\Xh)$ is isomorphic to $H^{2(|h|-k)}(\Xh)$ as $\mathfrak{S}_n$-modules for $0\le k\le |h|$. 
\end{prop}

\begin{proof}
Let $p\colon \Xh\to pt$ be the collapsing map to a point $pt$. Since $\Xh$ is a smooth projective variety, it has the canonical orientation induced from the complex structure and with this orientation, let 
\begin{equation} \label{eq:9-0}
p_!^{T}\colon H^*_{T}(\Xh)\to H^{*-2|h|}_{T}(pt)=H^{*-2|h|}(BT)
\end{equation}
be the equivariant Gysin homomorphism induced from the collapsing map $p$. As is well-known, $p_!^T$ is an $H^*(BT)$-module map and since $X(h)^T=\mathfrak{S}_n$, the Atiyah-Bott localization formula says that 
\begin{equation} \label{eq:9-1}
p_!^{T}(f)=\sum_{w\in \mathfrak{S}_n}\frac{f(w)}{e^T(w)}
\end{equation}
where $f(w)$ is the restriction of $f\in H^*_T(\Xh)$ to $w$ and $e^T(w)$ denotes the equivariant Euler class of the tangential $T$-module $\tau_w\Xh$ of $\Xh$ at $w$. Indeed, it is known that  
\[
\tau_w\Xh=\bigoplus_{j<i\le h(j)}\C(\pi_{w(i)}\pi_{w(j)}^{-1})\quad\text{as complex $\mathfrak{S}_n$-modules}
\] 
(in fact, Proposition~\ref{prop:equivariant_cohomology_Xh} follows from this fact), where $\pi_k\colon T\to S^1$ denotes the projection to the $(k,k)$-entry of the diagonal torus $T$ as before, so the equivariant Euler class of $\tau_w\Xh$ is given by 
\begin{equation} \label{eq:9-2}
e^T(w)=\prod_{j< i\le h(j)}(t_{w(i)}-t_{w(j)}).
\end{equation}

Since $p_!^T$ is an $H^*(BT)$-module homomorphism and 
\[
H^*(\Xh)=H^*_T(\Xh)/(t_1,\dots,t_n),
\]
$p_!^T$ induces an ordinary Gysin homomorphism 
\[
p_!\colon H^*(\Xh)\to H^{*-2|h|}(BT)/(t_1,\dots,t_n).
\]
This homomorphism is trivial unless $*=2|h|$. When $*=2|h|$, $p_!$ is nothing but the evaluation on the fundamental class of $\Xh$. Therefore, the bilinear form 
\[
\langle\ ,\ \rangle\colon H^{2k}(\Xh)\times H^{2|h|-2k}(\Xh)\to \C=H^0(BT)
\]
defined by 
\begin{equation} \label{eq:bilinear_form}
\langle f,g\rangle :=p_!(fg)
\end{equation}
is non-degenerate by Poincar\'e duality. 

\medskip
\noindent
{\bf Claim.} 
The non-degenerate bilinear form $\langle \ ,\ \rangle$ is $\mathfrak{S}_n$-invariant.

\medskip
\noindent
Indeed, for $\sigma\in \mathfrak{S}_n$, it follows from \eqref{eq:9-1} and \eqref{eq:9-2} that 
\begin{equation*}
\begin{split}
p_!^{T}(\sigma\cdot f)&=\sum_{w\in \mathfrak{S}_n}\frac{(\sigma\cdot f)(w)}{e^T(w)}=\sum_{w\in \mathfrak{S}_n}\frac{\sigma(f({\sigma^{-1}w}))}{e^T(w)}\\
&=\sum_{v\in \mathfrak{S}_n}\frac{\sigma(f(v))}{e^T({\sigma v})}=\sum_{v\in \mathfrak{S}_n}\frac{\sigma(f(v))}{\sigma(e^T(v))}
=\sigma\left(\sum_{v\in \mathfrak{S}_n}\frac{f(v)}{e^T(v)}\right)\\
&=\sigma(p_!^T(f)).
\end{split}
\end{equation*}
Therefore, the bilinear form $\langle \ , \ \rangle$ is $\mathfrak{S}_n$-invariant since $(\sigma \cdot f)(\sigma\cdot g)=\sigma\cdot(fg)$ and the $\mathfrak{S}_n$-action on $H^0(BT)=\Z$ is trivial, proving the claim. 

By the claim, $H^{2k}(\Xh)$ is isomorphic to the dual of $H^{2(|h|-k)}(\Xh)$ as $\mathfrak{S}_n$-modules. However, any $\mathfrak{S}_n$-module is isomorphic to its dual, so the proposition follows. 
\end{proof}

To prove the latter identity in \eqref{eq:palindromicity}, we also use the equivariant cohomology $H^*_T(\Xh)$. Since $H^*_T(\Xh)=H^*(BT)\otimes H^*(\Xh)$ as graded $\mathfrak{S}_n$-modules, we have 
\begin{equation} \label{eq:RHTXh}
R(H^*_T(\Xh);q)=R(H^*(BT);q)R(H^*(\Xh);q).
\end{equation}
Here, $H^*(BT)=\C[t_1,\dots,t_n]$ and $\mathfrak{S}_n$ permutes the indeterminates $t_1,\dots,t_n$, so the fundamental theorem of symmetric polynomials says that 
\begin{equation} \label{eq:HBT_invariant}
H^*(BT)^{\mathfrak{S}_n}=\C[e_1(t),\dots,e_n(t)],
\end{equation}
where $e_i(t)$ for $i=1,\dots,n$ denotes the $i$-th elementary symmetric polynomial in $t_1,\dots,t_n$. Moreover, $H^*(BT)$ is a free module over $H^*(BT)^{\mathfrak{S}_n}$ and 
\begin{equation} \label{eq:HBT}
H^*(BT)=H^*(BT)^{\mathfrak{S}_n}\otimes \Big(\C[t_1,\dots,t_n]/(e_1(t),\dots,e_n(t))\Big) 
\end{equation}
as graded $\mathfrak{S}_n$-modules. 

The factor in \eqref{eq:HBT}, which we denoted by 
\begin{equation} \label{eq:Nn}
N_n:=\C[t_1,\dots,t_n]/(e_1(t),\dots,e_n(t)),
\end{equation}
is of the same form as the cohomology ring of the flag variety $\Fln$:
\begin{equation} \label{eq:cohomology_flag}
H^*(\Fln)=\C[x_1,\dots,x_n]/(e_1(x),\dots,e_n(x))
\end{equation}
where $e_i(x)$ for $i=1,\dots,n$ denotes the $i$-th elementary symmetric polynomial in $x_1,\dots,x_n$. 
The dot action on $H^*(\Fln)$ is trivial (because so is on $x_i$'s) while $\mathfrak{S}_n$ permutes $t_i$'s as remarked above. Instead of the dot action, we consider the star action defined as 
\[
(\sigma*f)(w):=f(w\sigma)\quad\text{for $\sigma\in \mathfrak{S}_n$ and $f\in\Map(\mathfrak{S}_n,\C[t_1,\dots,t_n])$}.
\]
It follows from Proposition~\ref{prop:equivariant_cohomology_Xh} that the star action preserves the subring $H^*_T(\Fln)$. It also follows from \eqref{eq:x_i} and \eqref{eq:dot_action} that 
\begin{equation} \label{eq:star_action_on_t_and_x}
\sigma* t_i=t_{i},\quad \sigma* x_i=x_{\sigma(i)}\qquad \text{for $\sigma\in\mathfrak{S}_n$}.
\end{equation}
Therefore, the star action descends to $H^*(\Fln)$ by \eqref{eq:cohomology_flag} and $H^*(\Fln)$ with the star $\mathfrak{S}_n$-action is isomorphic to $N_n$ with the $\mathfrak{S}_n$-action permuting $t_i$'s as $\mathfrak{S}_n$-modules. 

With this understanding, we prove the following. 

\begin{lemm} \label{lemm:palindromicity}
Let $\mathbf{U}$ be the 
alternating complex $1$-dimensional $\mathfrak{S}_n$-module. 
Then 
\begin{enumerate}
\item $q^{n(n-1)/2}R(N_n;q^{-1})=R(N_n;q)\otimes\mathbf{U}$,
\item $R(H^*(BT);q^{-1})=(-q)^nR(H^*(BT);q)\otimes\mathbf{U}$.
\end{enumerate} 
\end{lemm}

\begin{proof}
(1) As we observed above, we may consider $H^*(\Fln)$ with the star $\mathfrak{S}_n$-action instead of $N_n$. As in the proof of Proposition~\ref{prop:palindromicity_Xh}, we have the non-degenerate bilinear form \eqref{eq:bilinear_form}. The claim in the proof of Proposition~\ref{prop:palindromicity_Xh} needs to be modified for the star action as follows. 

\medskip
\noindent
{\bf Claim.} $\langle \sigma*f,\sigma*g\rangle=\sgn(\sigma)\langle f,g\rangle$ for any $\sigma\in \mathfrak{S}_n$ and $f,g\in H^*(\Fln)$, where $\sgn(\sigma)$ denotes the sign of the permutation $\sigma$. 

\medskip
\noindent
Indeed, it follows from \eqref{eq:9-1} and \eqref{eq:9-2} that 
\begin{equation*}
\begin{split}
p_!^{T}(\sigma* f)&=\sum_{w\in \mathfrak{S}_n}\frac{(\sigma*f)(w)}{e^T(w)}=\sum_{w\in \mathfrak{S}_n}\frac{f(w\sigma)}{e^T(w)}\\
&=\sum_{v\in \mathfrak{S}_n}\frac{f(v)}{e^T({v\sigma^{-1}})}=\sum_{v\in \mathfrak{S}_n}\frac{f(v)}{(\sgn\sigma)e^T(v)}
=(\sgn\sigma)\left(\sum_{v\in \mathfrak{S}_n}\frac{f(v)}{e^T(v)}\right)\\
&=(\sgn\sigma)p_!^T(f)
\end{split}
\end{equation*}
where $e^T(v\sigma^{-1})=(\sgn \sigma)e^T(v)$ above follows from \eqref{eq:9-2} (note that $h(j)=n$ in \eqref{eq:9-2} for $\Xh=\Fln$). 
Therefore, the claim follows since $(\sigma *f)(\sigma*g)=\sigma*(fg)$ and the $\mathfrak{S}_n$-action on $H^0(BT)=\Z$ is trivial. 

Since any $\mathfrak{S}_n$-module is isomorphic to its dual, the claim above implies that $H^{2k}(\Fln)$ is isomorphic to $H^{n(n-1)-2k}(\Fln)\otimes\mathbf{U}$ as $\mathfrak{S}_n$-modules, proving (1). 

(2) It follows from \eqref{eq:HBT_invariant} and \eqref{eq:HBT} that
\begin{equation} \label{eq:BT_Nn}
R(H^*(BT);q)=R(H^*(BT)^{\mathfrak{S}_n};q)R(N_n;q)=\left(\prod_{k=1}^{n}\frac{1}{1-q^k}\right)R(N_n;q).
\end{equation}
Using this together with statement (1), we have 
\[
\begin{split}
R(H^*(BT);q^{-1})&=\left(\prod_{k=1}^{n}\frac{1}{1-q^{-k}}\right)R(N_n;q^{-1})\\
&=(-q)^n\left(\prod_{k=1}^{n}\frac{1}{1-q^k}\right)q^{n(n-1)/2}R(N_n;q^{-1})\\
&=(-q)^n\left(\prod_{k=1}^{n}\frac{1}{1-q^k}\right)R(N_n;q)\otimes\mathbf{U}\\
&=(-q)^nR(H^*(BT);q)\otimes\mathbf{U},
\end{split}
\]
proving (2). 
\end{proof}

Since
\[
R(H^*_T(\Xh);q)=R(H^*(BT);q)R(H^*(\Xh);q)
\]
by \eqref{eq:RHTXh}, the following corollary follows from Proposition~\ref{prop:palindromicity_Xh} and Lemma~\ref{lemm:palindromicity} (2).

\begin{coro} \label{coro:palindromicity_XhT}
$q^{|h|}R(H^*_T(\Xh);q^{-1})=(-q)^nR(H^*_T(\Xh);q)\otimes\mathbf{U}$. 
\end{coro}

This corollary implies the palindromicity of $\LLT$ polynomials in \eqref{eq:palindromicity}. Indeed, since 
\[
H^*_T(\Xh)=\C[x_1,\dots,x_n]\otimes \Y,
\]
we have 
\begin{equation} \label{eq:TXH_Y}
R(H^*_T(\Xh);q)=\frac{1}{(1-q)^n}R(\Y;q).
\end{equation}
Substituting this for the identity in Corollary~\ref{coro:palindromicity_XhT}, we obtain
\[
q^{|h|}\frac{1}{(1-q^{-1})^n}R(\Y;q^{-1})=(-q)^n\frac{1}{(1-q)^n}R(\Y;q)\otimes\mathbf{U},
\]
which reduces to 
\begin{equation} \label{eq:Y_theta}
q^{|h|}R(\Y;q^{-1})=R(\Y;q)\otimes\mathbf{U}.
\end{equation} 
Taking the Frobenius characteristic on both sides above, we obtain the desired palindromicity 
\[
q^{|h|}\LLT_h(\z;q^{-1})=\omega \LLT_h(\z;q)
\]
in \eqref{eq:palindromicity} by Proposition~\ref{prop:LLT} because the involution $\omega$ on $\Lambda_n$ corresponds to taking the tensor product with $\mathbf{U}$ in $\mathfrak{S}_n$-modules. 

The following is the twin version of Proposition~\ref{prop:palindromicity_Xh}. 

\begin{prop} \label{prop:palindromicity_Yh}
Let $\Yh$ be the twin of the regular semisimple Hessenberg variety $\Xh$ and we think of $H^*(\Yh)$ as an $\mathfrak{S}_n$-module with the dagger action. Then
\[q^{|h|}R(H^*(\Yh);q^{-1})=R(H^*(\Yh);q)\otimes \mathbf{U},
\]
in other words, $H^{2k}(\Yh)$ is isomorphic to $H^{2(|h|-k)}(\Yh)\otimes\mathbf{U}$ as $\mathfrak{S}_n$-modules for any $0\le k\le |h|$. 
\end{prop}

\begin{proof}
Since $\Y=H^*(\Yh)$ by Proposition~\ref{prop:L_Yh}, the proposition follows from the identity \eqref{eq:Y_theta}. 
\end{proof}

We give an explicit description of $R(N_n;q)$ later in Proposition~\ref{prop:R(N_n)} from which the following fact follows. Indeed, it should be well-known among experts. 

\begin{prop}[see \cite{tymo08}] 
\label{prop:Nn_regular}
$N_n$ with the $\mathfrak{S}_n$-action permuting $t_i$'s (in other words, $H^*(\Fln)$ with the star $\mathfrak{S}_n$-action) is a regular representation of $\mathfrak{S}_n$ if we forget the grading. 
\end{prop}

The $\mathfrak{S}_n$-module $H^*(\Yh)$ has the following nice property. The equivalent statement in terms of LLT polynomials is noticed in \cite[Lemma 4.1]{al-pa18}. 

\begin{prop} \label{prop:HYh_regular}
The $\mathfrak{S}_n$-module $H^*(\Yh)$ with the dagger action is a regular representation of $\mathfrak{S}_n$ for any Hessenberg function $h$ if we forget the grading. 
\end{prop}

\begin{proof}
Recall that $R(H^*_T(\Xh);q)$ has two presentations: one is 
\[
\begin{split}
R(H^*_T(\Xh);q)&=R(H^*(BT);q)R(H^*(\Xh);q)\\
&=\left(\prod_{k=1}^n\frac{1}{1-q^k}\right)R(N_n;q)R(H^*(\Xh);q)
\end{split}
\]
by \eqref{eq:RHTXh} and \eqref{eq:BT_Nn}, and the other is 
\[
R(H^*_T(\Xh);q)=\frac{1}{(1-q)^n} R(H^*(\Yh);q)
\]
by Proposition~\ref{prop:L_Yh} and \eqref{eq:TXH_Y}. Equating the right hand sides above, we obtain
\begin{equation} \label{eq:RYh}
R(H^*(\Yh);q)=\left(\prod_{k=1}^n\frac{1-q}{1-q^k}\right) 
R(N_n;q)R(H^*(\Xh);q).
\end{equation}
Therefore, setting $q=1$, we have
\begin{equation*} \label{eq:H*Yh}
\begin{split}
R(H^*(\Yh);1)&=(n!)^{-1}R(N_n;1)R(H^*(\Xh);1)\\
&=(n!)^{-1}R(N_n\otimes H^*(\Xh);1).
\end{split}
\end{equation*}
Here, $N_n\otimes H^*(\Xh)=n!N_n$ because $N_n$ is a regular representation of $\mathfrak{S}_n$ by Proposition~\ref{prop:Nn_regular} and $\dim_\C H^*(\Xh)=n!$, proving the proposition. 
\end{proof}

The cohomology $H^*(\Fln)$ has three $\mathfrak{S}_n$-actions: the dot action, the star action, and the dagger action. The dot action on $H^*(\Fln)$ is trivial.  The following proposition shows that $H^*(\Fln)$ with the other two actions are same as graded $\mathfrak{S}_n$-modules. 

\begin{prop} \label{prop:star&dagger}
$H^*(\Fln)$ with the dagger action agrees with $N_n$ as a graded $\mathfrak{S}_n$-module, so that it also agrees with $H^*(\Fln)$ with the star action as a graded $\mathfrak{S}_n$-module. 
\end{prop}

\begin{proof}
When $h=(n,\dots,n)$, $\Xh=\Fln$ and the dot action on its cohomology is trivial since the cohomology is generated by $x_1,\dots,x_n$ as a graded ring, so that $R(\Xh;q)$ for $h=(n,\dots,n)$ agrees with the Poincar\'e polynomial of $\Fln$, i.e. 
\[
R(H^*(\Xh);q)=\prod_{k=1}^n\frac{1-q^k}{1-q} \quad \text{for $h=(n,\dots,n)$}.
\]
Therefore, \eqref{eq:RYh} reduces to 
\[
R(H^*(\Yh);q)=R(N_n;q) \quad\text{for $h=(n,\dots,n)$}.
\] 
Here $\Yh$ for $h=(n,\dots,n)$ is also $\Fln$ but the dagger action is considered on its cohomology. This proves the proposition. 
\end{proof}

\section{Face modules of the permutohedron and LLT polynomials} \label{sect:6}

When $h=(2,3,\dots,n,n)$, the orbit space $\Xh/T=\Yh/T$ is the permutohedron $\Pi_n$ defined by 
\[
\Pi_n:=\text{Convex hull of }\{(w(1),\dots,w(n))\in \R^n\mid w\in \mathfrak{S}_n\}.
\]
It is a simple polytope of dimension $n-1$ and has symmetry of $\mathfrak{S}_n$. In this section, we will observe that the $\mathfrak{S}_n$-module $H^*(\Yh)$ can be interpreted in terms of $\mathfrak{S}_n$-modules generated by faces of $\Pi_n$. This gives a geometrical meaning of the $e$-positivity of $\LLT_h(\z;q+1)$. 

The $f$-polynomial of $\Pi_n$ and $h$-polynomial of $\Pi_n$ are respectively defined as 
\[
f_{\Pi_n}(q):=\sum_{i=0}^{n-1}f_i(\Pi_n)q^i,\qquad h_{\Pi_n}(q):=f_{\Pi_n}(q-1)
\]
where $f_i(\Pi_n)$ denotes the number of $i$-dimensional faces of $\Pi_n$. 
The symmetric group $\mathfrak{S}_n$ acts on $\R^n$ by permuting coordinates. This action preserves faces of $\Pi_n$, so the complex vector space $F_i(\Pi_n)$ generated by all $i$-dimensional faces of $\Pi_n$ becomes an $\mathfrak{S}_n$-module of dimension $f_i(\Pi_n)$ for each $i=0,1,\dots,n-1$. Therefore, 
\begin{equation} \label{eq:FH}
F_{\Pi_n}(q):=\sum_{i=0}^{n-1}F_i(\Pi_n)q^i,\qquad H_{\Pi_n}(q):=F_{\Pi_n}(q-1)
\end{equation}
are a natural generalization of $f_{\Pi_n}(q)$ and $h_{\Pi_n}(q)$ to $\mathfrak{S}_n$-modules. 

On the other hand, when $h=(2,3,\dots,n,n)$, $\Xh$ is a toric variety called a permutohedral variety. Its twin $\Yh$ is not a toric variety but a smooth $T$-manifold. Their Poincar\'e polynomials are same and known to agree with the $h$-polynomial $h_{\Pi_n}(q)$. 
Therefore, it is natural to ask whether either $R(H^*(\Xh);q)$ or $R(H^*(\Yh);q)$ agrees with $H_{\Pi_n}(q)$. The following gives the answer. 

\begin{theo} \label{theo:h=23nn}
Let $h=(2,3,\dots,n,n)$. Then 
\[
H_{\Pi_n}(q)=R(H^*(\Yh);q)\otimes\mathbf{U}.
\]
Indeed, these graded $\mathfrak{S}_n$-modules are given by 
\begin{equation} \label{eq:induction_description}
\sum_{I\subset [n-1]}\left(\Ind_{\mathfrak{S}_I}^{\mathfrak{S}_n}\mathbf{1}\right)(q-1)^{n-1-|I|}, 
\end{equation}
where $\mathbf{1}$ denotes the complex one-dimensional trivial $\mathfrak{S}_n$-module and for a subset $I=\{i_1,\dots,i_d\}$ of $[n-1]$ with $i_1<\cdots<i_d$, 
\begin{equation} \label{eq:SI}
\mathfrak{S}_I:=\mathfrak{S}_{i_1}\times \mathfrak{S}_{i_2-i_1}\times\cdots\times \mathfrak{S}_{n-i_d}.
\end{equation}
\end{theo}
\begin{proof}
First we shall observe that $H_{\Pi_n}(q)$ is given by \eqref{eq:induction_description}. For each proper nonempty subset $A$ of $[n]$, we define
\[
F_A:=\{(x_1,\dots,x_n)\in \Pi_n\mid \sum_{i\in I}x_i=|A|(|A|+1)/2\}.
\]
It is known that $F_A$ is a facet of $\Pi_n$ and any facet of $\Pi_n$ is of this form. The intersection $F_A\cap F_B$ is nonempty if and only if either $A\subset B$ or $B\subset A$. Since $\Pi_n$ is simple, this implies that any codimension $d$ face of $\Pi_n$ is of the form $\bigcap_{k=1}^d F_{A_k}$, where 
\[
\emptyset\subsetneq A_1\subsetneq A_2\subsetneq \cdots\subsetneq A_d\subsetneq [n].
\]
The symmetric group $\mathfrak{S}_n$ acts on the face $\bigcap_{k=1}^d F_{A_k}$ by 
\[
w\left(\bigcap_{k=1}^d F_{A_k}\right)=\bigcap_{k=1}^d F_{w(A_k)}\quad\text{for $w\in \mathfrak{S}_n$}.
\]

To each $I=\{i_1,\dots,i_d\}\subset [n-1]$ with $i_1<\cdots<i_d$, we associate a sequence
\[
\emptyset\subsetneq [i_1]\subsetneq [i_2]\subsetneq \cdots\subsetneq [i_d]\subsetneq [n]
\]
and hence a face $F(I):=\bigcap_{k=1}^d F_{[i_k]}$ of $\Pi_n$. 
The isotropy subgroup at the face $F(I)$ is $\mathfrak{S}_I$ and any face of $\Pi_n$ is in the $\mathfrak{S}_n$-orbit of $F(I)$ for some $I\subset [n-1]$. Since the cardinality $|I|$ is the codimension of $F(I)$, the dimension of $F(I)$ is $n-1-|I|$. 
This means that
\[
F_i(\Pi_n)=\sum_{I\subset [n-1],\ n-1-|I|=i}\Ind_{\mathfrak{S}_I}^{\mathfrak{S}_n}\mathbf{1}\quad\text{as $\mathfrak{S}_n$-modules.}
\]
Therefore, 
\begin{equation} \label{eq:FPi}
F_{\Pi_n}(q)=\sum_{i=0}^{n-1}F_i(\Pi_n)q^i=\sum_{I\subset [n-1]}\left(\Ind_{\mathfrak{S}_I}^{\mathfrak{S}_n}\mathbf{1}\right)q^{n-1-|I|}.
\end{equation}
Since $H_{\Pi_n}(q)=F_{\Pi_n}(q-1)$ by definition, $H_{\Pi_n}(q)$ is given by \eqref{eq:induction_description}. 

Now we shall observe that $R(H^*(\Yh);q)\otimes\mathbf{U}$ is given by \eqref{eq:induction_description}. 
The circle group $\Delta$ consisting of scalar matrices in the diagonal torus $T$ acts trivially on the flag variety $\Fln$ and the induced action of $T/\Delta$ on $\Xh$ is effective. Since $\Xh$ is a toric variety, the equivariant cohomology $H^*_{T/\Delta}(\Xh)$ is isomorphic to the face ring of $\Pi_n$. Therefore, each facet $F_A$ $(\emptyset\not=A\subsetneq [n])$ associates a degree two element $\tau_A$ in the equivariant cohomology ring and $\tau_A$'s generate the ring. Moreover, 
\begin{enumerate}
\item $\tau_{A_1}\cdots \tau_{A_d}\not=0$ \text{if and only if} \text{$\bigcap_{k=1}^d F_{A_k}\not=\emptyset$},
\item monomials in $\tau_A$'s form an additive basis of $H^*_{T/\Delta}(\Xh)$, 
\item $w\cdot \tau_{A}=\tau_{w(A)}$ for $w\in \mathfrak{S}_n$. 
\end{enumerate} 

For a subset $I=\{i_1,\dots,i_d\}$ of $[n-1]$ with $i_1<\cdots<i_d$, we define
\[
\tau(I):=\tau_{[i_1]}\cdots \tau_{[i_d]}.
\]
Then, it follows from the observation above that the isotropy subgroup at $\tau(I)$ is $\mathfrak{S}_I$ and we have a decomposition 
\begin{equation*} \label{eq:HTDelta_Xh}
H^*_{T/\Delta}(\Xh)=\bigoplus_{I\subset [n-1]}\Ind_{\mathfrak{S}_I}^{\mathfrak{S}_n}\Big(\C[\tau_{[i_1]},\dots,\tau_{[i_d]}]\tau(I)\Big)\quad\text{as $\mathfrak{S}_n$-modules,}
\end{equation*}
where we understand $\C[\tau_{[i_1]},\dots,\tau_{[i_d]}]\tau(I)=\C$ when $I=\emptyset$. It follows that 
\begin{equation} \label{eq:RHTDelta}
R(H^*_{T/\Delta}(\Xh);q)=\sum_{I\subset [n-1]}\frac{q^{|I|}}{(1-q)^{|I|}}\Ind_{\mathfrak{S}_I}^{\mathfrak{S}_n}\mathbf{1}.
\end{equation}

Since the action of $\Delta$ on $\Xh$ is trivial, we have $H^*_T(\Xh)=H^*_{T/\Delta}(\Xh)\otimes H^*(B\Delta)$ and the $\mathfrak{S}_n$-action on $H^*(B\Delta)=\C[t_1+\cdots+t_n]$ is trivial. Therefore, it follows from \eqref{eq:RHTDelta} that 
\[
\begin{split}
R(H^*_T(\Xh);q)=&R(H^*_{T/\Delta}(\Xh);q)R(H^*(B\Delta);q)\\
=&\left(\sum_{I\subset [n-1]}\frac{q^{|I|}}{(1-q)^{|I|}}\Ind_{\mathfrak{S}_I}^{\mathfrak{S}_n}\mathbf{1}\right)\frac{1}{1-q}\\
=&\frac{1}{(1-q)^n}\sum_{I\subset [n-1]}q^{|I|}(1-q)^{n-1-|I|}\Ind_{\mathfrak{S}_I}^{\mathfrak{S}_n}\mathbf{1}.
\end{split}
\]
Since $h=(2,3,\dots,n,n)$, we have $|h|=n-1$; so the above description together with the duality in Corollary~\ref{coro:palindromicity_XhT} yields
\[
\begin{split}
(-q)^nR(H^*_T(\Xh);q)\otimes\mathbf{U}&=q^{n-1}R(H^*_T(\Xh);q^{-1})\\
&=\frac{q^{n-1}}{(1-q^{-1})^n}\sum_{I\subset [n-1]}q^{-|I|}(1-q^{-1})^{n-1-|I|}\Ind_{\mathfrak{S}_I}^{\mathfrak{S}_n}\mathbf{1}\\
&=\frac{q^n}{(q-1)^n}\sum_{I\subset [n-1]}(q-1)^{n-1-|I|}\Ind_{\mathfrak{S}_I}^{\mathfrak{S}_n}\mathbf{1}.
\end{split}
\] 
Dividing the both sides above by $(-q)^n$, we obtain
\begin{equation} \label{eq:RHTXH_theta}
R(H^*_T(\Xh);q)\otimes\mathbf{U}=\frac{1}{(1-q)^n}\sum_{I\subset [n-1]}(q-1)^{n-1-|I|}\Ind_{\mathfrak{S}_I}^{\mathfrak{S}_n}\mathbf{1}.
\end{equation}
On the other hand, we have
\begin{equation} \label{eq:RHT_q_Yh}
R(H^*_T(\Xh);q)=\frac{1}{(1-q)^n}R(H^*(\Yh);q)
\end{equation}
by Proposition~\ref{prop:L_Yh}. Comparing \eqref{eq:RHTXH_theta} with \eqref{eq:RHT_q_Yh}, we obtain 
\begin{equation*} \label{eq:Yh_and_face_Sn_module}
R(H^*(\Yh);q)\otimes\mathbf{U}=\sum_{I\subset [n-1]}(q-1)^{n-1-|I|}\Ind_{\mathfrak{S}_I}^{\mathfrak{S}_n}\mathbf{1},
\end{equation*}
proving the desired fact. 
\end{proof}

Since 
\[
\LLT_h(\z;q)=\sum_{i=0}^\infty \ch(H^{2i}(\Yh)q^i
\]
by Proposition~\ref{prop:L_Yh}, Theorem~\ref{theo:h=23nn} says that when $h=(2,3,\dots,n,n)$, we have 
\begin{equation} \label{eq:LLT_h_polynomial}
\LLT_h(\z;q)=\sum_{I\subset [n-1]}\ch\left(\Ind_{\mathfrak{S}_I}^{\mathfrak{S}_n}\mathbf{U}\right)(q-1)^{n-1-|I|}.
\end{equation}
Since 
\[
\mathfrak{S}_I=\mathfrak{S}_{i_1}\times \mathfrak{S}_{i_2-i_1}\times\cdots\times \mathfrak{S}_{n-i_d}
\]
for $I=\{i_1,\dots,i_d\}$ with $i_1<\cdots<i_d$, we have 
\begin{equation} \label{eq:ePI}
\ch\left(\Ind_{\mathfrak{S}_I}^{\mathfrak{S}_n}\mathbf{U}\right)=e_{i_1}e_{i_2-i_1}\cdots e_{n-i_d}=e_{P(I)}
\end{equation}
by definition of the Frobenius characteristic $\ch$, where $e_i$ denotes the $i$-th elementary symmetric function and $P(I)$ denotes the partition $(i_1,i_2-i_1,\dots,n-i_d)$ of $n$ associated to $I$. Therefore, we obtain the following corollary from \eqref{eq:LLT_h_polynomial}, \eqref{eq:ePI}, and \eqref{eq:FPi}. It gives a geometrical meaning of $\LLT_h(\z;q+1)$ for $h=(2,3,\dots,n,n)$. 

\begin{coro} \label{coro:LLT(q+1)}
When $h=(2,3,\dots,n,n)$, we have
\[
\LLT_h(\z;q+1)=\sum_{I\subset [n-1]}q^{n-1-|I|}e_{P(I)}=\sum_{i=0}^{n-1}\ch(F_i(\Pi_n)\otimes\mathbf{U})q^i
\]
where $F_i(\Pi_n)$ is the $\mathfrak{S}_n$-module generated by $i$-dimensional faces of $\Pi_n$, see \eqref{eq:FH}. 
\end{coro}

\begin{rema}
The presence of the factor $\mathbf{U}$ in $F_i(\Pi_n)\otimes \mathbf{U}$ suggests to consider an $\mathfrak{S}_n$-module $\widetilde{F}_i(\Pi_n)$ generated by \emph{oriented} $i$-dimensional faces of $\Pi_n$. Indeed, we orient faces of $\Pi_n$ as follows, where an orientation of a face is a choice of an equivalence class of bases for a vector space tangent to the face.    
First, we fix an orientation of $\Pi_n$.  To each facet $F_A$ where $A$ is a subset of $[n]$, let $\mathbf{n}_A$ be the inward orthonormal vector to $\Pi_n$ and we give an orientation $o_{F_A}$ to $F_A$ such that $(\mathbf{n}_A,o_{F_A})$ agrees with the orientation of $\Pi_n$.  Any codimension $d$-face $F$ of $\Pi_n$ is the intersection of $d$ facets, i.e.\ $F=\bigcap_{j=1}^d F_{A_j}$ where
\[
\emptyset\subsetneq A_1\subsetneq A_2\subsetneq \cdots\subsetneq A_d\subset [n]
\]
and we give an orientation $o_F$ to $F$ such that $(\mathbf{n}_{A_1},\dots,\mathbf{n}_{A_d},o_F)$ agrees with the orientation of $\Pi_n$. The action of $\sigma\in \mathfrak{S}_n$ on $\Pi_n$ maps $\mathbf{n}_{A_1},\dots,\mathbf{n}_{A_d}$ to  $\mathbf{n}_{\sigma(A_1)},\dots,\mathbf{n}_{\sigma(A_d)}$ in this order, so $\sigma$ preserves the orientations of faces if and only if it preserves the orientation of $\Pi_n$, and this is equivalent to $\sgn(\sigma)=1$ because the normal vector $(1,\dots,1)$ to $\Pi_n$ in $\R^n$ is fixed under the action of $\sigma$ and $\sigma$ permutes the coordinates of $\R^n$. Therefore, $\widetilde{F}_i(\Pi_n)$ is isomorphic to $F_i(\Pi_n)\otimes\mathbf{U}$ as $\mathfrak{S}_n$-modules.  
\end{rema}

Finally, we shall observe that the former identity in Corollary~\ref{coro:LLT(q+1)} agrees with the combinatorial $e$-expansion formula of $\LLT_h(\z;q+1)$ in \cite{al-pa18, al-su22}. In order to state the $e$-expansion formula, we need some notation. 

Let $G_h$ be the unit interval graph associated to $h$, i.e.\ the vertex set is $[n]$ and $\{i,j\}$ is an edge of $G_h$ if and only if $j\le h(i)$ (when $i<j$) or $i\le h(j)$ (when $j<i$).
Let $\theta$ be an orientation of the graph $G_h$, that assigns to each edge $\{i,j\}$ a directed edge $\overrightarrow{ij}$ or $\overrightarrow{ji}$.
We say that a directed edge $\overrightarrow{ij}$ is ascending in $\theta$ if $i<j$ and is descending in $\theta$ if $i>j$. 
For a vertex $i$ of $G_h$, the \emph{highest reachable vertex}, denoted ${\rm hrv}(\theta,i)$, is defined as the maximal $j$ such that there is a directed path from $i$ to $j$ in $\theta$ using only ascending edges.
The orientation $\theta$ defines a set partition $\pi(\theta)$ of the vertices $[n]$ of $G_h$, where two vertices belong to the same block if and only if they have the same highest reachable vertex.
Finally, let $\lambda(\theta)$ denote the partition given by the sizes of the blocks in $\pi(\theta)$. 

\begin{theo}[\cite{al-su22}] \label{theo:al-su}
Let $h$ be any Hessenberg function on $[n]$, $O(h)$ the set of all orientations of the unit interval graph $G_h$, and ${\rm asc}(\theta)$ the number of ascending edges in $\theta\in O(h)$. Then 
\begin{equation} \label{eq:al-su}
\LLT_h(\z;q+1)=\sum_{\theta\in O(h)}q^{{\rm asc}(\theta)}e_{\lambda(\theta)}.
\end{equation}
\end{theo}

In our case where $h=(2,3,\dots,n,n)$, the unit interval graph $G_h$ is the path graph $P_n$ with vertices $[n]$ in an increasing order.
To each $I=\{i_1,\dots,i_d\}$ with $i_1<\cdots<i_d$, we assign an orientation $\theta_I$ of $P_n$ such that an edge $\{i,i+1\}$ $(i\in [n-1])$ is ascending unless $i=i_k$ for some $k=1,\dots,d$. Then 
\[
{\rm asc}(\theta_I)=n-1-|I|
\]
and for any vertex $i$ of $P_n$, the highest reachable vertex is given by 
\[
{\rm hrv}(i,\theta_I)=i_k \quad\text{if $i_{k-1}<i\le i_k$}
\] 
where $i_0=0$ and $i_{d+1}=n$. Therefore, 
\[
\lambda(\theta_I)=(i_1,i_2-i_1,\dots,n-i_d)=P(I).
\]
Since the correspondence $I\to \theta_I$ gives a bijection between subsets of $[n-1]$ and orientations of $P_n$, the observation above shows that the former identity in Corollary~\ref{coro:LLT(q+1)} coincides with \eqref{eq:al-su} when $h=(2,3,\dots,n,n)$.

\section{Complete graph} \label{sect:7}

When $h=(n,\dots,n)$, i.e.\ $h(j)=n$ for any $j\in [n]$, the unit interval graph $G_h$ is a complete graph $K_n$ with $[n]$ as vertices. In this section, we confirm Theorem~\ref{theo:al-su} when $h=(n,\dots,n)$ from our viewpoint. We begin with 

\begin{prop} \label{prop:R(N_n)}
Let $N_n$ be the graded $\mathfrak{S}_n$-module defined in \eqref{eq:Nn}. Then 
\[
\begin{split}
R(N_n,\q)&=\sum_{I\subset [n-1]}\Big(\prod_{i\in I}\q^i\prod_{j\in [n-1]\backslash I}(1-\q^j)\Big)\Ind_{\mathfrak{S}_I}^{\Sn}\mathbf{1}\\
&=\sum_{I\subset [n-1]}\Big(\prod_{j\in [n-1]\backslash I}(\q^j-1)\Big)\Ind_{\mathfrak{S}_I}^{\Sn} \mathbf{U}.
\end{split}
\]
In particular, $N_n$ is the regular representation of $\mathfrak{S}_n$ if we forget the grading. 
\end{prop}

\begin{proof}
First we find $R(H^*(BT),\q)$. Remember that $H^*(BT)=\C[t_1,\dots,t_n]$ and $w\cdot t_i=t_{w(i)}$ for $w\in \mathfrak{S}_n$. Therefore, any monomial in $\C[t_1,\dots,t_n]$ is in the $\mathfrak{S}_n$-orbit of a monomial of the form 
\begin{equation} \label{eq:10-1}
t_1^{a_1}t_2^{a_2}\cdots t_n^{a_n}\quad (a_1\ge a_2\ge \cdots \ge a_n\ge 0).
\end{equation}
Any $a=(a_1,a_2,\dots,a_n)$ above is of the form 
\[
a_1=\dots=a_{i_1}> a_{i_1+1}=\dots=a_{i_2}>\dots>a_{i_{d-1}+1}=\dots=a_{i_d}>a_{i_d+1}=\dots=a_n
\]
for some $1\le i_1<i_2<\cdots<i_d\le n-1$, and we define
\[
I(a):=\{i_1,i_2,\dots,i_d\} \subset [n-1]. 
\]
For instance, if all $a_i$'a are distinct, then $I(a)=[n-1]$,
and if all $a_i$'s are same, then $I(a)=\emptyset$.
Then the isotropy subgroup at the element $t_1^{a_1}\cdots t_n^{a_n}$ of the $\mathfrak{S}_n$-action on $H^*(BT)=\C[t_1,\dots,t_n]$ is $\mathfrak{S}_{I(a)}$, so the $\mathfrak{S}_n$-module generated by the $\mathfrak{S}_n$-orbit of $t_1^{a_1}\cdots t_n^{a_n}$ is $\Ind_{\mathfrak{S}_{I(a)}}^{\mathfrak{S}_n}\mathbf{1}$.

We compute the Hilbert series of the vector space spanned by elements $t_1^{a_1}\cdots t_n^{a_n}$ with $I(a)=I=\{i_1,\dots,i_d\}$. Although the cohomology degree of each $t_i$ is $2$, we think of $t_i$ as degree $1$ for simplicity. The Hilbert series of the vector space is computed as follows: 
\begin{align*}
	&\sum_{I(a)=I}\q^{a_1+\dots+a_n}\\
	=& \sum_{a_{i_1}>a_{i_2}>\dots>a_{i_d}>a_n\ge 0}\q^{i_1a_{i_1}+(i_2-i_1)a_{i_2}+\dots+(i_d-i_{d-1})a_{i_d}+(n-i_d)a_n}\\
	=& \sum_{a_{i_2}>\dots>a_{i_d}>a_n\ge 0}\Big(\sum_{a_{i_1}=a_{i_2}+1}^\infty \q^{i_1a_{i_1}}\Big)\q^{(i_2-i_1)a_{i_2}+\dots+(i_d-i_{d-1})a_{i_d}+(n-i_d)a_n}\\
	=&\sum_{a_{i_2}>\dots>a_{i_d}>a_n\ge 0}\Big(\frac{\q^{i_1(a_{i_2}+1)}}{1-\q^{i_1}}\Big)\q^{(i_2-i_1)a_{i_2}+\dots+(i_d-i_{d-1})a_{i_d}+(n-i_d)a_n}\\
	=&\frac{\q^{i_1}}{1-\q^{i_1}}\sum_{a_{i_2}>a_{i_3}>\dots>a_{i_k}>a_n\ge 0}\q^{i_2a_{i_2}+(i_3-i_2)a_{i_3}+\dots+(i_d-i_{d-1})a_{i_d}+(n-i_d)a_n}\\
	=& \dots\\
	=&\Big(\prod_{j=1}^d\frac{\q^{i_j}}{1-\q^{i_j}}\Big)\sum_{a_n\ge 0}\q^{na_n}
	=\frac{1}{1-\q^n}\prod_{j=1}^d\frac{\q^{i_j}}{1-\q^{i_j}}
	=\frac{1}{1-\q^n}\prod_{i\in I}\frac{\q^{i}}{1-\q^{i}}.
\end{align*}
Therefore
\[
R(H^*(BT),\q)=\frac{1}{1-\q^n}\sum_{I\subset [n-1]}\Big(\prod_{i\in I}\frac{\q^{i}}{1-\q^{i}}\Big)\Ind_{\mathfrak{S}_I}^{\Sn}\mathbf{1}.
\]
On the other hand, we have 
\[
R(H^*(BT),\q)=\Big(\prod_{i=1}^n\frac{1}{1-\q^i}\Big)R(N_n,\q)
\]
by \eqref{eq:BT_Nn}. Comparing the two descriptions of $R(H^*(BT);\q)$ above, we obtain 
\[
R(N_n,\q)=\sum_{I\subset [n-1]}\Big(\prod_{i\in I}\q^i\prod_{j\in [n-1]\backslash I}(1-\q^j)\Big)\Ind_{\mathfrak{S}_I}^{\Sn}1
\]
which is the former identity in the proposition. 

The latter identity in the proposition follows from the former identity proved above and the palindromicity 
\[
q^{n(n-1)/2}R(N_n;q^{-1})=R(N_n;q)\otimes\mathbf{U}
\]
in Lemma~\ref{lemm:palindromicity} (1). 

The last statement in the proposition can be seen as follows. 
We set $q=1$ in the former identity in the proposition. Then the products at the right hand side of the identity vanish unless $I=[n-1]$. Therefore, we have 
\[
N_n=R(N_n;1)=\Ind_{\mathfrak{S}_{[n-1]}}^{\mathfrak{S}_n}\mathbf{1}.
\]
Here, $\mathfrak{S}_{[n-1]}=\mathfrak{S}_1\times \cdots\times \mathfrak{S}_1$ by definition \eqref{eq:SI}, so $\Ind_{\mathfrak{S}_{[n-1]}}^{\mathfrak{S}_n}\mathbf{1}$ is the regular representation of $\mathfrak{S}_n$. 
\end{proof}

By Proposition~\ref{prop:L_Yh} we have 
\[
\LLT_h(\z;q)=\sum_{i=0}^\infty \ch(H^{2i}(\Yh)q^i
\]
and $H^*(\Yh)$ for $h=(n,\dots,n)$ is the same as $N_n$ as graded $\mathfrak{S}_n$-module by Proposition~\ref{prop:star&dagger}. Therefore, the following corollary follows from Proposition~\ref{prop:R(N_n)}.

\begin{coro} \label{coro:LLTnn}
Let $h = (n,\dots,n)$. Then 
\begin{equation} \label{eq:LLTnn}
\LLT_h(\z;q+1)=\sum_{I\subset [n-1]}\left(\prod_{j\in [n-1]\backslash I}((q+1)^j-1)\right)e_{P(I)}
\end{equation}
where $P(I)$ denotes the partition $(i_1,i_2-i_1,\dots,n-i_d)$ of $n$ associated to $I=\{i_1,\dots,i_d\}$ with $i_1<\cdots<i_d$ as before. 
\end{coro}

This corollary manifestly shows the $e$-positivity of $\LLT_h(\z;q+1)$ for $h=(n,\dots,n)$. 
In order to observe the agreement of \eqref{eq:LLTnn} and \eqref{eq:al-su} for $h=(n,\dots,n)$, we prepare an elementary lemma. 

\begin{lemm} \label{lemm:elementary}
Let $1\le i<n$ be positive integers. Then 
\begin{equation} \label{eq:elementary}
\sum_{1\le a_1<\cdots<a_i<n}\ \prod_{j=1}^ix^{a_j-j}=\prod_{j=1}^i\frac{x^{n-j}-1}{x^j-1}
\end{equation}
where $a_1,\dots,a_i$ run over all positive integers with $1\le a_1<\cdots<a_i<n$. 
\end{lemm}

\begin{proof}
We prove \eqref{eq:elementary} by induction on $n$. When $n=2$, the both sides in \eqref{eq:elementary} are $1$, so it holds in this case. 

Suppose that \eqref{eq:elementary} holds for $n-1$. We split the left hand side in \eqref{eq:elementary} into two parts according as $a_i=n-1$ or $a_i<n-1$, i.e. 
\[
x^{n-1-i}\sum_{1\le a_1<\cdots<a_{i-1}<n-1}\ \prod_{j=1}^{i-1}x^{a_j-j}+\sum_{1\le a_1<\cdots<a_{i}<n-1}\ \prod_{j=1}^ix^{a_j-j}.
\]
Applying the inductive assumption to these two sums, the above turns into 
\[
\begin{split}
&
x^{n-1-i}\prod_{j=1}^{i-1}\frac{x^{n-1-j}-1}{x^j-1}+\prod_{j=1}^i\frac{x^{n-1-j}-1}{x^j-1}\\
=&\left(\prod_{j=1}^{i-1}\frac{x^{n-1-j}-1}{x^j-1}\right)\left(x^{n-1-i}+\frac{x^{n-1-i}-1}{x^i-1}\right)
=\prod_{j=1}^{i}\frac{x^{n-j}-1}{x^j-1},
\end{split}
\]
completing the induction step. 
\end{proof}

Now we shall observe the agreement of \eqref{eq:LLTnn} and \eqref{eq:al-su}. Recall that since $h=(n,\dots,n)$, the unit interval graph $G_h$ is the complete graph $K_n$ with vertices $[n]$. 

{\bf Step 1.} We first observe the agreement of the coefficients of $e_n$. 
As for \eqref{eq:al-su}, $\lambda(\theta)=n$ if and only if each vertex of $K_n$ has an ascending path to the vertex $n$ in the orientation $\theta$. The latter is equivalent to the condition that each vertex $j$ different from $n$ has an ascending edge to at least one vertex in $n-j$ vertices $j+1,\dots,n$. Therefore
\begin{equation} \label{eq:en}
\sum_{\theta\in O(h),\ \lambda(\theta)=n}q^{\asc(\theta)}=\prod_{j=1}^{n-1}\left((q+1)^{n-j}-1\right)=\prod_{j=1}^{n-1}\left((q+1)^{j}-1\right).
\end{equation} 
On the other hand, as for \eqref{eq:LLTnn}, $P(I)=n$ if and only if $I=\emptyset$. Therefore, 
the coefficient of $e_n$ in \eqref{eq:LLTnn} agrees with the last term in \eqref{eq:en}.

{\bf Step 2.} We fix a subset $I=\{i_1,\dots,i_d\}$ of $[n-1]$ with $i_1<i_2<\cdots<i_d$. The case $I=\emptyset$ treated in Step 1 can be regarded as the case where $d=0$, so we assume $d\ge 1$. The ordered set 
\[
P(I)=(i_1,i_2-i_1,\cdots,i_d-i_{d-1},n-i_d)
\]
gives a partition of $n$. We prove that the coefficient of $e_{P(I)}$ in \eqref{eq:al-su} agrees with 
\begin{equation*}
\sum_{\theta\in O(h),\ \lambda(\theta)=P(I)}q^{\asc(\theta)}
\end{equation*}
by induction on $d$. We have observed the agreement for $d=0$ in Step 1. Since $\lambda(\theta)=P(I)$, $\pi(\theta)$ consists of $d+1$ ordered blocks $(\pi_0(\theta),\dots,\pi_d(\theta))$ where the largest vertex in $\pi_r(\theta)$ is smaller than that in $\pi_s(\theta)$ if and only if $r<s$. 

Applying the inductive assumption to the first $d$ blocks $(\pi_0(\theta),\dots,\pi_{d-1}(\theta))$ and to the last block $\pi_d(\theta)$ respectively, we have 
\begin{equation} \label{eq:step3}
\sum_{\theta\in O(h),\ \lambda(\theta)=P(I)}q^{\asc(\theta)}=A\prod_{j\in [i_d-1]\backslash I}\left((q+1)^j-1\right)\prod_{j=1}^{n-i_d-1}\left((q+1)^j-1\right)
\end{equation}
where $A$ is the term coming from directed edges between vertices in $\bigcup_{k=0}^{d-1}\pi_k(\theta)$ and $\pi_d(\theta)$. 
We note that an edge between $a\in \bigcup_{k=0}^{d-1}\pi_k(\theta)$ and $b\in \pi_d(\theta)$ must be directed as $\overleftarrow{ab}$ if $a<b$ but both $\overrightarrow{ab}$ and $\overleftarrow{ab}$ can be allowed if $a>b$. If $\bigcup_{k=0}^{d-1}\pi_k(\theta)=\{a_1,\dots,a_{i_d}\}$ with $a_1<\cdots<a_{i_d}$, then for each $a_j$ $(j=1,\dots,i_d)$ there are $a_j-j$ elements in $\pi_d(\theta)$ which are smaller than $a_j$. Therefore
\[
A=\sum_{1\le a_1<\cdots<a_{i_d}<n}\ \prod_{j=1}^{i_d}(q+1)^{a_j-j} 
\]
where $a_1,\dots,a_{i_d}$ run over all positive integers with $1\le a_1<\cdots<a_{i_d}<n$, and applying Lemma~\ref{lemm:elementary} to the right hand side above, we obtain 
\[
A=\prod_{j=1}^{i_d}\frac{(q+1)^{n-j}-1}{(q+1)^j-1}
\]
since the numbers of vertices in $\bigcup_{k=1}^{d-1}\pi_k(\theta)$ and $\pi_d(\theta)$ are respectively $i_d$ and $n-i_d$. Plugging the above $A$ in \eqref{eq:step3}, \eqref{eq:step3} reduces to 
\[
\sum_{\theta\in O(h),\ \lambda(\theta)=P(I)}q^{\asc(\theta)}=\prod_{j\in [n-1]\backslash I} \left((q+1)^j-1\right).
\]
Here the right hand side above agrees with the coefficient of $e_{P(I)}$ in \eqref{eq:LLTnn}. Thus the agreement of \eqref{eq:al-su} and \eqref{eq:LLTnn} has been confirmed for $h=(n,\dots,n)$. 

\bigskip
\noindent
{\bf Acknowledgment.} We thank Hiraku Abe for pointing out that the quotient ring \eqref{eq:intro_quotient_Yh} agrees with the cohomology ring of the twin $\Yh$.  After writing this paper, we learned from Tatsuya Horiguchi that Precup-Sommers also notice the relation \eqref{eq:intro_Yh}, see \cite[Appendix A]{pr-so22}. 
We thank him for informing us of the paper.

\end{document}